\documentclass[a4paper,reqno]{amsart}
\numberwithin{equation}{section}

\newtheorem{theorem}{Theorem}[section]

\newtheorem{lemma}[theorem]{Lemma}

\newtheorem{corollary}[theorem]{Corollary}

\def\beq{\begin{equation}}
\def\eeq{\end{equation}}
\def\be{\begin{equation*}}
\def\ee{\end{equation*}}

\usepackage{amssymb}
\usepackage{amsfonts}

\title{Dedekind-finite Cardinals having countable partitions}
\author{Supakun Panasawatwong and J K Truss} 
\date{This paper is an amplified version of part of the first author's PhD thesis at the University of Leeds}

\begin{document}

\begin{abstract} We study the possible structures which can be carried by sets which have no countable subset, 
but which fail to be `surjectively Dedekind finite', in two possible 
senses, that there is surjection to $\omega$, or alternatively, that there is a surjection to a proper superset.
\end{abstract}

\maketitle  

\setcounter{footnote}{1}\footnotetext{
keywords: Dedekind finite, weakly Dedekind finite, tree}
\newcounter{number}

\section {Introduction}

In \cite{Panasawatwong} a study was made of a long list of possible definitions 
of `finiteness', following on from earlier work on this, for instance in \cite{Levy},
\cite{Truss1}, \cite{Truss3}, \cite{Degen}, \cite{Goldstern}. A set is said to be {\em Dedekind finite} 
if it has no countably infinite subset, or equivalently, there is no 
bijection to a proper subset. This notion is vacuous in the presence of the axiom of 
choice, AC, where this is just the same a saying that the set is finite 
(i.e. has cardinality in $\omega$), so except when we want to construct 
certain models, we shall not assume AC.

The aim in \cite{Panasawatwong} was to try to unify all the approaches mentioned in a systematic way, from the 
point of view of inclusions between the classes so defined, and their closure under the 
most natural operations such as unions and products. It was found that there was a natural 
subdivision, into those which admit no surjection to $\omega$, called in \cite{Walczak} 
`weakly Dedekind finite', and those which have no countable subset, but do admit a surjection
to $\omega$. Cardinalities of sets for which there is no countable subset (`Dedekind finite')
were written in \cite{Truss1} as $\Delta$, and those of sets for which there is no surjection 
to $\omega$ as $\Delta_4$. Another class was introduced there, written $\Delta_5$, comprising
the cardinalities of sets having no surjection to a proper superset. The analogy between the two classes
can be expressed by saying that for $|X|$ to lie in $\Delta$, any injection from $X$ to $X$ must 
also be surjective, whereas for it to lie in $\Delta_5$, any surjection from $X$ to $X$ must be 
injective. The most stringent notion (apart from actual finiteness) is called being `amorphous',
which we mention in some places. A set is said to be {\em amorphous} if it is infinite, but cannot be
written as the disjoint union of two infinite sets. Such a set can actually carry quite a lot of
structure (despite the name), and a detailed study was carried out in \cite{Truss3}.

Some tree structures which arise naturally when considering weak versions of the axiom of choice related
to K\"onig's Infinity Lemma were studied in \cite{Truss2} and \cite{Forster}. The trees in question are
`balanced' trees of height $\omega$, meaning that they have $\omega$ levels, and on each level the 
ramification order is constant (in this case, finite). For such a tree there is a natural surjection
to itself, obtained by mapping each non-minimal node to its predecessor, clearly surjective but not 
injective, so its cardinality cannot lie in $\Delta_5$. In a suitable model, it will however lie in 
$\Delta$. 

The main focus of this paper is therefore on trees which may or may not arise on a Dedekind finite set.
We begin by looking at cardinals in $\Delta_5 \setminus \Delta_4$, where trees as just mentioned do not 
arise. However, other types of trees, called `weakly 2-transitive' may, and this is a good source of
examples. We study members of $\Delta_5 \setminus \Delta_4$, some of which can be written as a 
countable union of weakly Dedekind finite sets, and others, more typically, cannot. We give 
$2^{\aleph_0}$ inequivalent examples, in a sense of `equivalence' introduced in \cite{Truss4}.

The other main case examined is of cardinals in $\Delta \setminus \Delta_5$, which is done in
terms of trees of height $\omega$, which, it is shown, must arise in this situation. We show that under
natural hypotheses, any such tree has a balanced subtree, and we extend the work of \cite{Forster} in 
the finitely branching case, and give some generalizations for infinitely branching trees.

\section {Preliminaries}

The main focus of the paper will be on the relationship between tree structures that a set can carry, and notions 
of Dedekind-finiteness. Here by a {\em tree} we understand a partially ordered set $(T, \le)$ in which any two 
elements have a common lower bound, and for every $x \in T$, $\{y \in T: y \le x\}$ is linearly ordered (under the 
restriction of $\le$). We consider two main cases, the first being densely ordered, as in \cite{Droste1} and 
\cite{Droste2}, appropriate for studying members of $\Delta_5 \setminus \Delta_4$ (sometimes $\Delta_4$), and 
the other well-founded, which arises when we wish to study members of $\Delta \setminus \Delta_5$.  In the 
latter case there is a unique least element (the `root') and for each $x \in T$, $\{y \in T: y \le x\}$ is 
well-ordered. For any such tree there are `levels' $L_\alpha$ given recursively by $L_\alpha$ is the set of minimal elements of 
$T \setminus \bigcup_{\beta < \alpha}L_\beta$ (which implies that $L_0$ just consists of the root). A maximal 
element of called a `leaf'. Since $T$ is assumed to be a set, there is a least ordinal for which $L_\alpha = \emptyset$, 
which is the {\em height} of the tree. For us, well-founded trees will always have height $\omega$, since any 
greater height would violate Dedekind-finiteness.

\begin{lemma} For any set $X$, $|X| \not \in \Delta_5$ if and only if there is a subset $T$ of $X$ which carries a 
well-founded tree structure of height $\omega$ and no leaves. \label{2.1}  \end{lemma}

\begin{proof}  Suppose $|X| \not \in \Delta_5$. Then there is a surjection $f$ from $X$ to $X$ which is not injective. Let $f$ map
distinct $x$ and $y$ to the same point. Then the restriction $g$ of $f$ to $X \setminus \{x\}$ is a surjection from a proper
subset of $X$ to $X$. Let $L_n = g^{-n}(x)$ for $n \in \omega$, and let $T = \bigcup_{n \in \omega}L_n$. Since $x \not \in {\rm dom}(g)$,
members of $L_n$ are those $a$ such that $n$ is least such that $g^n(a) = x$, and it follows that the $L_n$ are pairwise disjoint. We
let $a \le b$ if $g^n(b) = a$ for some $n$, and since $g^{-1}(a)$ is nonempty, there are no leaves. Clearly $L_n$ is the $n$th level of
the tree, so $T$ has $\omega$ levels.

Conversely, if a $T$ as stated exists, then this gives rise to a surjective but not injective function from $X$ to $X$ by mapping each
element of $T$ which is not the root to the immediately preceding element (which exists because it lies in a successor level, so there is a 
unique greatest point below it) and fixing all other points (including the root).   \end{proof}

There is a related but weaker notion, which is that $|X| \in \Delta_5^*$ if there is no finite-to-one surjection from $X$ to $X$
which is not injective. By adapting the above proof, the following results.

\begin{lemma} For any set $X$, $|X| \not \in \Delta_5^*$ if and only if there is a subset $T$ of $X$ which carries a tree structure 
of height $\omega$, finite levels, and no leaves. \label{2.2}  \end{lemma}

\begin{corollary} $\Delta_5 \subseteq \Delta_5^*$. \label{2.3}  \end{corollary}
\begin{proof}  This is immediate from the definitions, or else one can use Lemmas \ref{2.1} and \ref{2.2}.  \end{proof}

\section{Constructions of cardinals in $\Delta_5 \setminus \Delta_4$}

We briefly review the Fraenkel--Mostowski method, which we shall use, and which presents fewer technical complications
than forcing. We work in FMC, which is the same as ordinary ZFC set theory, except that 
we allow a set $U$ of `atoms' (`urelemente'), being objects which have no elements, but are different from the empty set. 
This can be axiomatized by replacing the axiom of extensionality by an axiom which says that no member of $U$ has 
any elements, and sets not lying in $U$ having the same members are equal. One starts with a ground model
$\mathfrak M$ containing an interpretation for $U$, together with a group $G$ of permutations of $U$, and a `normal'
filter $\mathfrak F$ of subgroups of $G$ (meaning that it is closed under conjugacy), all lying in $\mathfrak M$. We allow 
$G$ to act on the whole of $\mathfrak M$ by defining $g(x) = \{g(y): y \in x\}$, and then the standard notions of 
pointwise and setwise stabilizers $G_x$ and $G_{\{x\}}$ of $x \in {\mathfrak M}$ make sense. Thus $G_x$ is the
set of elements $g$ of $G$ such that $g(y) = y$ for every $y \in x$, and $G_{\{x\}}$ is the set of elements $g$ of $G$ 
such that $g(x) = x$. The Fraenkel--Mostowski model thereby defined is 
${\mathfrak N} = \{x \in {\mathfrak M}: G_{\{x\}} \in {\mathfrak F} \wedge x \subseteq {\mathfrak N}\}$ (which is a valid 
definition by transfinite induction). Normality of $\mathfrak F$ is required so that the axiom of
replacement holds in the model. Usually it is assumed that $G_{\{u\}} \in {\mathfrak F}$ for each $u \in U$, in which 
case $U \in {\mathfrak N}$. We say that $\mathfrak F$ is `generated by finite supports' if it is the family of
all subgroups of $G$ containing the pointwise stabilizer of a finite subset of $U$. This is automatically normal,
since $gG_Ag^{-1} = G_{gA}$. If $G_{\{x\}} \ge G_A$, we say that $x$ is {\em supported by} $A$.

It can be checked that $\mathfrak N$ satisfies all the axioms of FMC, except the axiom
of choice, AC. This provides a conceptually relatively straightforward method for producing models in which AC
is false, and the method predates forcing. Cohen showed how to adapt the main ideas to give models of ZF in which
choice is false, and the Jech--Sochor Theorem \cite{Jech} provides general circumstances in which consistency results achieved
using FM models can be automatically transferred to forcing proofs.

We now use the Fraenkel--Mostowski method to present various ways in which cardinals in $\Delta_5 \setminus \Delta_4$ 
can be constructed. The first family of examples provides the desired set rather directly, with the surjection to $\omega$
included explicitly. The most well-known model of this type gives Russell's `pairs of socks', meaning that there is a
countable sequence of pairs without a choice function. Thus we let the set of atoms in the ground model $\mathfrak M$
be $U = \{u_{ni}: n \in \omega, i \in 2\}$. For ease we write $U_n = \{u_{n0}, u_{n1}\}$, and the group $G$ used to 
define the model is taken to be the set of permutations of $U$ fixing each $U_n$. The filter $\mathfrak F$ of subgroups of
$G$ is taken to be generated by finite supports, which is automatically normal,
and hence gives rise to a model $\mathfrak N$ of FM. Clearly the partition into the sets $U_n$ persists in the model. To
see that $|U| \in \Delta_5$ in $\mathfrak N$, suppose that $f \in {\mathfrak N}$ is a surjection from $U$ to 
$U \cup \{0\}$. Then $f$ must be supported by $\bigcup_{k < N}U_k$ for some $N$. For $i \ge N$, $f^{-1}U_i$ must be 
non-empty. If $u_{mj} \in f^{-1}U_i \setminus U_i$ then we can interchange the two members of $U_i$ leaving $u_{mj}$
and all members of $\bigcup_{k < N}U_k$ fixed, contrary to $f$ a function supported by $\bigcup_{k < N}U_k$. Hence
$f^{-1}U_i \subseteq U_i$. It follows that $f$ maps $\bigcup_{k < N}U_k$ onto a set containing 
$\bigcup_{k < N}U_k \cup \{0\}$, which is impossible for a finite set.  

We can generalize this example to the following case. Let ${\mathcal A}_n$ be non-trivial $\aleph_0$-categorical structures 
having pairwise disjoint domains, and assume that they are all transitive, meaning that their automorphism groups $G_n$
act transitively. Then we just copy the above construction, taking $U_n$ to be indexed by the members of ${\mathcal A}_n$. 
The automorphism group $G$ of the resulting set $U$ of atoms is taken to be the (unrestricted) direct product 
of the $G_n$. By the results of \cite{Walczak}, in the resulting model $\mathfrak N$, each $U_n$ has cardinality in
$\Delta_4$. The fact that $|U| \in \Delta_5$ follows by the same argument as for `pairs of socks'. For suppose that $f$ 
is a map from $U$ onto $U \cup \{0\}$, having finite support contained in $\bigcup_{k < N}U_k$. If for some $i \ge N$,
there is $x \in f^{-1}U_i \setminus U_i$, then again by non-triviality the member $f(x)$ of $U_i$ can be moved by a member 
of $G$ while keeping fixed all members of $\{x\} \cup \bigcup_{k < N}U_k$, contrary to $f$ a function supported by 
$\bigcup_{k < N}U_k$. Hence for each $i \ge N$, $f^{-1}U_i \subseteq U_i$, and so $f$ maps $\bigcup_{k < N}U_k$ onto 
a set containing the proper superset $\bigcup_{k < N}U_k \cup \{0\}$. However, $\Delta_4$ is closed under forming 
finite unions, so this gives a contradiction.

We remark that actually what is required to make the above argument work is that each $G_n$ acts transitively on $U_n$,
and $|U_n| \in \Delta_5$. Requiring $U_n$ to arise from an $\aleph_0$-categorical structure corresponds to 
$|U_n| \in \Delta_4$, so that is a stronger hypothesis than necessary, though easier to describe. For instance, we can
take each $U_n$ to arise from a weakly 2-transitive tree, as described below, and although these need not be 
$\aleph_0$-categorical, the argument to show that $|U| \in \Delta_5$ still goes through.

Note that if the requirement of transitivity is dropped, then $\bigcup_{n \in \omega}U_n$ may even fail to lie in
$\Delta$. For instance, if every ${\mathfrak A}_n$ is a copy of the closed rational interval $[0, 1]$ then the set of
left endpoints forms a countable subset of $U$.

Since $|U| \not \in \Delta_4$, by \cite{Walczak}, the corresponding structure is not $\aleph_0$-categorical, so it must
have non-isomorphic but elementarily equivalent models. What these are will depend on the precise choice of language
used to describe the structure. For instance if we axiomatize Russell's pairs of socks using unary predicates 
${\underline P}_n$ to stand for $\{u_{n0}, u_{n1}\}$, then any model is determined up to isomorphism by the 
cardinality of the set of realizations of the non-principal type $\{\neg {\underline P}_n: n \in \omega\}$, so there
are many options. These are not of particular interest however from the point of view of models for Dedekind-finite
cardinals. If we instead try axiomatizing the pairs of socks by means of a partial order in which
$u_{mi} < u_{nj} \Leftrightarrow m < n$, then in non-standard models any `infinite' point lies in a copy of 
${\mathbb Z} \times \{0, 1\}$, so these models have a very different character. The moral of these examples is that
ordinary first order logic is not the right way to describe non-weakly Dedekind finite sets, and a suitably
strengthened language should be used instead.

We now demonstrate by means of some examples, that there are many sets having cardinality in $\Delta_5$ which
cannot be written as a countable union of weakly Dedekind finite sets, so the situation so far described is
very much atypical. For a start, we can vary the pairs of socks example by having a longer transfinite sequence 
of pairs. In the resulting model, $U$ has not only countably infinite partitions, but it may also have well-ordered 
partitions of greater cardinality. In the simplest case, we can let $U = \{u_{\alpha i}: \alpha < \omega_1, i < 2\}$, 
$G$ be the group of all permutations of $U$ fixing each $U_\alpha = \{u_{\alpha 0}, u_{\alpha 1}\}$ (setwise), and 
$\mathfrak F$ be generated by finite supports. The proof that in the resulting FM model, $|U| \in \Delta_5$ is a
slightly simpler version of that for the next model. To see that $U$ cannot be written as a countable union of 
weakly Dedekind-finite sets, observe that any weakly Dedekind-finite subset of $U$ must actually be finite, as
if it intersects infinitely many $U_\alpha$ then it can be mapped onto $\omega$, and of course $U$ is not a countable
union of finite sets.

A superficially similar example, which is however radically different, is as follows. Let $U$ and $U_\alpha$ be 
as in the previous example, but this time we take the group $G$ of all permutations of $U$ which preserve the
partition $\Pi = \{U_\alpha: \alpha < \omega_1\}$, and for the filter $\mathfrak F$ of subgroups we take the
family of all subgroups containing the pointwise stabilizer $G_{A \cup B}$ of sets $A \cup B$ for which $A$ is a
finite subset of $U$, and $B$ is a countable subset of $\Pi$. This is a normal filter since 
$gG_{A \cup B}g^{-1} = G_{gA \cup gB}$. Each countable subset of $\Pi$ then lies in the resulting model
$\mathfrak N$, and so $|U| \not \in \Delta_4$. But we can see that $|U| \in \Delta_5$ as follows. Suppose
that $f$ is a surjection from $U$ to $U \cup \{0\}$, and let $G_{\{f\}} \ge G_{A \cup B}$. Let us write $X_n = f^{-n}(0)$.
Since $f$ is surjective, all the $X_n$ are disjoint and non-empty. Since $A$ is finite, there is $\alpha$ such that
$u_{\alpha 0}, u_{\alpha 1} \not \in A$ and either one or both of $u_{\alpha 0}$, $u_{\alpha 1}$ lie in some $X_n$. First
treating the case in which they both lie in some $X_n$, let $u_{\alpha 0} \in X_m$ and $u_{\alpha 1} \in X_n$, and 
assume without loss of generality that $m \le n$. Let $\pi \in G$ interchange $u_{\alpha 0}$ and $u_{\alpha 1}$ and 
fix all other points. Then $\pi \in G_{A \cup B}$, so $\pi(f) = f$. Since $f$ is surjective, there is 
$u_{\beta i}$ mapped by $f$ to $u_{\alpha 1}$, and $u_{\beta i} \in X_{n+1}$, so 
$u_{\beta i} \neq u_{\alpha 0}, u_{\alpha 1}$. Therefore $\pi$ fixes $u_{\beta i}$. Now 
$\langle u_{\beta i}, u_{\alpha 1} \rangle \in f$, so as $\pi$ fixes $f$, 
$\langle u_{\beta i}, u_{\alpha 0} \rangle \in f$ too. But this contradicts $f$ a function. If just one of 
$u_{\alpha 0}$, $u_{\alpha 1}$ lies in some $X_n$, suppose for example that it is $u_{\alpha 0}$. Consider
$\pi$ swapping $u_{\alpha 0}$ and $u_{\alpha 1}$ and fixing all other points. If 
$f(u_{\beta i}) = u_{\alpha 0}$ then $u_{\beta i} \not \in U_\alpha$, so it is fixed by $\pi$, and as in the first
case we contradict $f$ a function. As in the previous
example, $U$ is not a countable union of weakly Dedekind-finite sets. We remark that the partition $\Pi$ is, in
$\mathfrak N$, a `quasi-amorphous' set, in the sense defined in \cite{Creed}. This means that it is uncountable,
every infinite subset has a countable subset, but it cannot be written as the disjoint union of two uncountable sets.

More complicated examples arise from 2-transitive or weakly 2-transitive trees, as defined in \cite{Droste1} and 
\cite{Droste2}. We recall the basic definitions. A tree is said to be {\em 2-transitive} if for any two 
isomorphic 2-element substructures there is an automorphism which takes the first to the second (which is not
required to extend the given isomorphism). It is {\em weakly 2-transitive} if for any two 2-element chains
there is an automorphism taking the first to the second. (The difference therefore is that the automorphism
group of a weakly 2-transitive tree is not required to act transitively on 2-element antichains.) Now these 
two differ in that countable 2-transitive trees are $\aleph_0$-categorical, but ones which
are weakly 2-transitive need not be. In fact in the most typical cases, a weakly 2-transitive tree exhibits 
infinitely many distinct ramification orders, so is definitely not $\aleph_0$-categorical. We assume non-triviality, 
namely that all maximal chains are isomorphic to $\mathbb Q$, and that there are incomparable points.

In order to explain `ramification', we require a tree, written $T^+$, which is obtained from $T$ by adjoining
extra points so that any two members of $T$ have a greatest lower bound (meet) in $T^+$. A construction for 
this is given for example in \cite{Droste1}. It can then be checked that any two members of $T^+$ also have a 
meet in $T^+$, so $T^+$ is the least lower semilattice extending $T$, and all the `new' points of $T^+$ are of the 
form $a \wedge b$ for some $a, b \in T$. A {\em ramification point} is then a member of $T^+$ which is the meet of two 
incomparable members of $T$. At each ramification point $x$, there is a natural equivalence relation on points
above $x$ under which $y \sim z$ if for some $t > x$, $y, z \ge t$. The equivalence classes are called {\em cones}
at $x$, and the number of cones is called the {\em ramification order} of $x$. In the 2-transitive case, the 
ramification order of all ramification points is the same, but this is not true for weakly 2-transitive trees.

Some complications are caused by `special' ramification points, being those at least one of whose cones has a least
element. If this arises, then the maximal chains of $T^+$ will not be dense, but will contain consecutive pairs, of
which the upper point lies in $T$ but the lower does not. In particular we need to take account of what we call
{\em exceptional} ramification points, being those special ones which have just one cone having a least element. If 
a special ramification point has more than one cone with a least element, then by weak 2-transitivity these cones can
be interchanged. To make life easier in Theorem \ref{3.2} we shall assume that there are no special ramification 
points at all, but for our first result in this section, Theorem \ref{3.1} we do treat the general case.

If $A$ is a finite subset of a weakly 2-transitive tree then we write $[A]$ for the set of all members of $T^+$ fixed
by all elements of $G_A$. This is also called the {\em definable closure} of $A$. It may be explicitly described as 
the lower subsemilattice of $T^+$ generated by $A$, together with the least points of cones at any exceptional 
ramification points arising. This is also finite, and in fact is equal to $A \cup \{a \wedge b: a, b \in A\}$ together with 
the least points of cones at exceptional ramification points of the form $a \wedge b$ for $a, b \in A$. 

Now consider an FM-model induced by a weakly 2-transitive tree $T$. Let the set of atoms be $U = U_T = \{u_t: t \in T\}$.
Let $G$ be the group of automorphisms of $U$ induced by the automorphism group of $T$, and let ${\mathfrak N}_T$ be 
the corresponding FM-model with finite supports. We carry over the $[A]$ notation to this situation. We sometimes 
adjoin $-\infty$ below the tree, fixed by all automorphisms.

\begin{theorem}  \label{3.1} For any countable weakly $2$-transitive tree $T$, $|U_T| \in \Delta_5$ in ${\mathfrak N}_T$, 
and if $T$ has infinitely many ramification orders, then $|U_T| \not \in \Delta_4$.   \end{theorem}

\begin{proof} Let $f: U_T \to U_T \cup \{*\}$ where $* \not \in U_T$ be surjective in ${\mathfrak N}_T$ with finite 
support $A$.

We remark that for any $x, y$, if $f(x) = y$, then $y \in [A \cup \{x\}]$. For since $A$ supports $f$, $\pi(f) = f$ 
for all $\pi \in G_A$. Then $\langle x, y \rangle \in f$ and so $\langle \pi x, \pi y \rangle \in \pi(f) = f$. Hence if 
$\pi y \neq y$, then $\pi x \neq x$ since $f$ is a function. If $y \not \in [A \cup \{x\}]$, then by definition of 
definable closure $[ \;\;]$, there is $\pi \in G_{A \cup \{x\}} \subseteq G_A$ such that $\pi y \neq y$ but $\pi x = x$, 
contrary to what we have shown. Hence $y \in [A \cup \{x\}]$. Also note that $[A \cup \{x\}]$ is the union of 
$[A] \cup \{x\}$ and $\{x \wedge a: a \in A\}$, together with the minimal points of cones at exceptional ramification 
points in this set. 

Now consider $y \not \in [A]$ and let $x$ be such that $f(x) = y$. We aim to show that $x = y$. Now $x \in [A]$ would imply 
that $y \in [A]$, contrary to supposition, and we deduce that $x \not \in [A]$.

\noindent{\bf Case 1}: $y \not \le a$ for all $a \in [A]$.

Then $y \neq x \wedge a$ for all $a \in A$, and also $y$ is not the minimal point of a cone at an exceptional ramification 
point of $[A]$ (as this would lie in $[A]$). So if $y \neq x$, then $y$ must be the minimal point of a cone at an exceptional
ramification point of the form $x \wedge a$ for some $a \in A$. 

Let $z$ be such that $f(z) = x$. If $x \le b$ for some $b \in [A]$, then $x \wedge a = b \wedge a$, which is impossible.
Applying the same argument as above for $x$ and $z$, $x$ must be the minimal point of a cone at an exceptional ramification
point of the form $z \wedge b$ for some $b \in A$. But now again $x \wedge a = b \wedge a$ which is impossible.

The conclusion is that $x = y$.

\noindent{\bf Case 2}: $y \le a$ for some $a \in [A]$.

If $x \not \le b$ for all $b \in [A]$, then by the argument in Case 1, $f(x) = x$, and $x = y$ again follows. Hence $x \le b$
for some $b \in A$. Then for any $c \in A$, either $x \wedge c = x$ or $x \wedge c \in [A]$. Since $x \not \in [A]$, no new
ramification points are added in passing from $[A]$ to $[A \cup \{x\}]$, and therefore $[A \cup \{x\}] = [A] \cup \{x\}$. It 
follows that $x = y$. 

From both cases, we have $f(y) = y$ for all $y \not \in [A]$. Hence $f[A] = [A] \cup \{*\}$ which is a 
contradiction since $f$ is surjective but $[A]$ is finite. Therefore such $f$ does not exist
in ${\mathfrak N}_T$ and so $|U_T| \in \Delta_5$ in ${\mathfrak N}_T$.  

The final clause follows from the fact that $T$ is not $\aleph_0$-categorical.  \end{proof}

We are now able to deduce that there are many `essentially distinct' examples of this type. The sense in which
we mean essentially distinct was introduced in \cite{Truss4}, in terms of a notion called `equivalence'. We say 
that sets $X$ and $Y$ are {\em equivalent}, written $X \equiv Y$, if for any first-order structure in a countable 
language which has $X$ as its domain, there is an elementarily equivalent first-order structure having domain $Y$, 
and vice versa. In the presence of AC, any two infinite sets are equivalent (as follows from the L\"owenheim--Skolem 
Theorems), but without choice this may not be the case. The intuition is that a set may hide some structure, which is 
obscured in the presence of the ability to well-order everything. For instance, if $X$ is Dedekind-finite, then so 
is $Y$, as was remarked in \cite{Truss4}. 

Let us first see that if $|X| \in \Delta_5$ and $X \equiv Y$, then $|Y| \in \Delta_5$. For if not, there is a
surjective but not injective function $f: Y \to Y$. In a first order language containing a function symbol for $f$,
one can express `$f$ is a surjective but not injective function', and this is true in $(Y, f)$. Since $X \equiv Y$,
there must be an interpretation for the function symbol making the same statement true in $X$, contrary to 
$|X| \in \Delta_5$.

Next we can show that there are $2^{\aleph_0}$ inequivalent examples of sets in $\Delta_5 \setminus \Delta_4$ arising 
from weakly 2-transitive trees. It is easiest for this purpose to consider just those weakly 2-transitive trees
in which the maximal chains in $T^+$ are densely ordered. The main complication in the classification given in
\cite{Droste2} comes about when this is {\em not} the case, which happens if there are special ramification points,
ones having a cone with a least element. A general description of all countable weakly 2-transitive trees requires 
listing information about special ramification points. Since there are $2^{\aleph_0}$ cases in which there are no 
special ramification points, we shall just deal with the case in which there are none. To specify the countable 
weakly 2-transitive tree having no special ramification points up to isomorphism, it suffices to state which 
ramification orders arise, and whether or not the points of $T$ themselves ramify (with their ramification order 
also given). Thus we require an infinite subset of $\{2, 3, \ldots, \infty\}$, and a statement of the ramification 
order of points of $T$ (which will be 1 if they do not ramify, and some member of
$\{2, 3, \ldots, \infty\}$ otherwise). We call this list the {\em code} of $T$.

\begin{theorem}  \label{3.2} For any countable weakly $2$-transitive trees $T_1$, $T_2$ in which all maximal chains
of $T_1^+$, $T_2^+$ are dense, having distinct infinite codes, $|U_{T_1}|$ and $|U_{T_2}|$ are $\equiv$-inequivalent 
members of $\Delta_5 \setminus \Delta_4$.   \end{theorem}

\begin{proof}  The point of the restriction we have made on our weakly 2-transitive trees is that for any
ramification point $r$ and $x, y > r$ in $T$, there is an automorphism of $T$ fixing $r$ and taking $x$ to $y$.

Suppose for a contradiction that $|U_{T_1}| \equiv |U_{T_2}|$. In ${\mathfrak N}_{T_1}$ there must therefore be a 
structure $\prec$ on $U_{T_1}$ such that $(U_{T_1}, \prec)$ in ${\mathfrak N}_{T_1}$ is elementarily equivalent to 
$(U_{T_2}, <)$ in ${\mathfrak N}_{T_2}$. Let $A$ be a finite support for $(U_{T_1}, \prec)$ in ${\mathfrak N}_{T_1}$. 
Thus any automorphism of $(U_{T_1}, <)$ that fixes $A$ pointwise also fixes $\prec$. We may assume that $A$ is a 
subtree of $U_{T_1}$.

In order to handle this situation, we use the following notation, where $A$ is a finite subset of $T$ and $[A]$ 
is its definable closure, in this case just the lower subsemilattice generated by $A$. If $a$ is maximal in $A$ 
(or if $A = \emptyset$, $a$ can be $-\infty$), we let $U_a = \{x \in T: a < x\}$, and if $a < b$ are consecutive 
members of $[A]$, then $L_{ab} = \{x \in T: a < x < b\}$ (the linear piece between $a$ and $b$; again here $a$ 
can be $-\infty$ in the case when $b$ is minimal in $A$), and $S_{ab} = \{x \in T: a < x, b \not \le x\}$ (the 
corresponding side piece). Now $L_{ab}$ is linear, and we may also consider $L_{ab}^+$, which is the same thing, 
but calculated in $T^+$, that is, including ramification points too, which are coloured according to which orbit 
they lie in. If $x \in L_{ab}^+$, then $C_x = \{y \in T: b \wedge y = x\}$, which is the set of points which 
branch off from $L_{ab}^+$ at $x$. This is a union of cones at $x$.

A key remark is that the pointwise stabilizer $G_A$ acts transitively on each $U_a$, $L_{ab}$, and $C_x$ (here heavily 
using the assumption that the maximal chains in $T^+$ are dense). But more is true. 
In fact for any $b, c \in U_a$, there is an automorphism of $T$ taking $b$ to $c$ which fixes $A$ pointwise, {\em and} has
support contained in $U_a$, with a similar statement holding for $C_x$. The analogous statement also holds for
weak 2-transitivity. For instance, if $a < b < c$ and $a < d < e$, then by weak 2-transitivity there is an 
automorphism $f$ fixing $a$ and taking $b$ to $d$, and applying weak 2-transitivity again, there is an automorphism
$g$ fixing $d$ and taking $f(c)$ to $e$, and $f$, $g$ may be chosen fixing all points below $a$ and $d$ respectively.
Then $gf$ fixes $a$ and takes $b$ to $d$ and $c$ to $e$.

Although not actually needed, we can describe the orbits of $G_A$. They are of the form $U_a$, and $L_{ab}$, as well 
as the union of the cones in $S_{ab}$ which meet $L_{ab}$ at points lying in some orbit of $G_A$ on $L_{ab}^+$.

Now being a tree with dense chains is first order expressible, and so this must be true of $(U_{T_1}, \prec)$. We shall show
that on each set of the form $U_a$ or $C_x$, $\prec$ is equal to $<$ or the empty set (in which case it would be an antichain, 
i.e. pairwise $\prec$-incomparable). We just treat $U_a$, as $C_x$ is similar.

First consider a maximal point $a$ of $A$ (and let $a$ have a default value of $-\infty$ if $A = \emptyset$). First we note
that if $b, c \in U_a$ are $<$-incomparable, then they are also $\prec$-incomparable. This is because there is an 
element of $G_A$ swapping $b$ and $c$. Since $G_A$ preserves $\prec$, $b \prec c \Leftrightarrow c \prec b$, and since
$\prec$ is antisymmetric, $\neg b \prec c \wedge \neg c \prec b$. 

Next we show that there cannot be points $b, c \in U_a$ such that $b < c$ and $c \prec b$. For, if so, by weak
2-transitivity of $G_A$ on $U_a$, for $x, y \in U_a$, $x < y \Rightarrow y \prec x$. Choose $z > a$ and incomparable 
$x, y > z$. Thus $x, y \prec z$. By the previous paragraph, $x$ and $y$ are also $\prec$-incomparable, so this 
contradicts $(U_{T_1}, \prec)$ a tree. 

Next suppose that there are $b, c \in U_a$ such that $b < c$, and $b$ and $c$ are $\prec$-incomparable. By 
weak 2-transitivity of $U_{T_1}$, $a < x < y$ implies that $x$, $y$ are $\prec$-incomparable. But also as shown above, if
$b$ and $c$ are $<$-incomparable, they are also $\prec$-incomparable, and we deduce that $\prec$ is the empty relation
on $U_a$.

Otherwise, it follows that for $x, y \in U_a$, $x < y \Rightarrow x \prec y$. Since also $x, y$ $<$-incomparable
$\Rightarrow x, y \prec$-incomparable and $y < x \Rightarrow y \prec x$, we deduce that $x < y \Leftrightarrow x \prec y$. 
In other words, $<$ and $\prec$ agree on $U_a$.

Now let us consider $S_{ab}$ where $a < b$ are consecutive members of $A$. First note that by a proof similar to the above,
if $x, y$ are $<$-incomparable members of $C_c$ for $a < c < b$, then they are also $\prec$-incomparable, since there is a 
member of $G_A$ which swaps $x$ and $y$ (and fixes $c$). We can deduce that if $a < c < d < b$, and $x \in C_c$, $y \in C_d$,
then $x$ and $y$ are $\prec$-incomparable. For suppose that $x \prec y$ (a similar argument applying if $y \prec x$). Choose
$x' \in C_c$ which is $<$-incomparable with $x$. Then by what we have just shown, $x$, $x'$ are $\prec$-incomparable. 
Furthermore, there is a member of $G_A$ which swaps $x$ and $x'$, and having support contained in $C_c$, and so this also
fixes $y$. Since $G_A$ preserves $\prec$, $x' \prec y$, but this contradicts $\prec$ a tree.

It follows that any $\prec$-comparabilities in $S_{ab}$ must hold between members of the same $C_c$. Suppose then that
$x < y$ and $x \prec y$ hold for some $x, y \in C_c$. As $G_A$ acts weakly 2-transitively on $C_c$, $x < y \Rightarrow x \prec y$. 
Since we have also shown that $x$, $y$ $<$-incomparable implies that they are $\prec$-incomparable, we deduce that $<$ and $\prec$
agree on $C_c$. 

The main remaining point is to show that there is some $a$ or $c$ such that $\prec$ is non-empty on $U_a$ or $C_c$. Suppose not. Then
each $U_a$ is an antichain, and since the members of distinct $C_c$s for $a < c < b$ are incomparable, each $S_{ab}$ is also an
antichain. Consider the possible relationship between $a \in A$ and members of $U_b$ under $\prec$. The group $G_A$ acts transitively 
on $U_b$, and so if any member of $U_b$ is $\prec$-less than $a$, they all are, which violates $\prec$ a tree. Similarly, for members
of $C_c$. The same argument applies to the possible relationship between members of distinct $U_a$s, or $C_c$s, or between members 
of $U_a$ and $C_c$. We deduce that the union of all the $U_a$s and $S_{ab}$s is an antichain, and all its members are either 
incomparable with every member of $A$, or above some member of $a$. However, since $(U_{T_1}, \prec)$ is a tree satisfying the
same first order sentences as $(U_{T_1}, <)$, all its chains are infinite, so this is impossible.

We deduce that there is $a$ or $c$ such that $<$ and $\prec$ agree on $U_a$ or $C_c$. We just treat the first case, and show that 
for any $b \in U_a$, for $x \in T_1$, $b < x \Leftrightarrow b \prec x$. From left to right is already known. Suppose for a 
contradiction that there are some $b \in U_a$ and $x$ such that $b \prec x$ but not $b < x$. Since $<$ and $\prec$ agree on $U_a$, 
$x \not \in U_a$. Choose $c \in U_a$ incomparable with $b$. Then there is a member of $G_A$ taking $b$ to $c$ and with support 
contained in $U_a$. This preserves $\prec$, and fixes $x$, and hence $c \prec x$. Since $\prec$ is a tree, $b$ and $c$ are 
$\prec$-comparable, and as they both lie in $U_a$, they are also $<$-comparable, which gives a contradiction.

Finally we choose any $b \in U_a$, and observe that the ramification orders arising above $b$ in $<$ and $\prec$ are equal to those
arising in $T_1$ and $T_2$ respectively. Since these are determined by first order formulae, it follows that $|U_{T_1}|$ and $|U_{T_2}|$ 
have equal codes after all.  \end{proof}

\section{Balanced trees}

To study Dedekind finite sets not lying in $\Delta_5$ by using FM constructions, in view of Lemma \ref{2.1}, 
we are naturally led to consider possible tree structures with $\omega$ levels and no leaves. We note that 
we didn't say that our $X$ was Dedekind-finite, though that is actually what we had in mind. To achieve this, 
it is convenient to insist at the very least, if we are trying to use the tree as the set of atoms in a 
Fraenkel--Mostowski model, that the group of automorphisms of the tree act transitively on each level. This 
is captured by the idea of `balanced tree', which was introduced in \cite{Forster}, and can be extended. We 
start by recalling the simplest case, of finite levels. A classical result about this situation is K\"onig's 
Lemma, which says that an infinite tree with $\omega$ levels, all finite, has an infinite branch (maximal 
linearly ordered subset). This requires the axiom of choice in its proof, and in \cite{Forster} and 
\cite{Truss2} the versions of the axiom of choice needed for various special cases are studied. 
   
A finite-branching tree $T$ is called {\em balanced} if the sets of immediate successors of each vertex 
$x$, denoted by ${\rm succ}(x)$, in the same level are of equal cardinality.  The case that $T$ branches 
infinitely is more complicated so we shall deal with that in the later part of this section.

For now, by {\em subtree} of a tree we understand a downwards closed subset under the induced 
ordering. Note that this is also a tree, and the level that an element of the subtree lies in does 
not alter from the original tree.

\begin{lemma} Any tree $(T, \le)$ with $\omega$ levels, all finite, has a balanced subtree, having no leaves,
and also having $\omega$ levels. \label{4.1}  \end{lemma}

\begin{proof} Let $S$ be a sequence of natural numbers such that every number occurs infinitely often, 
say $S = \langle k_n: n \in \omega \rangle$. We construct a decreasing sequence $T_n$ of subtrees 
of $T$, each having $\omega$ levels, such that for each $n$, $T_n$ is pruned on level $k_n$ so that 
every member has the same degree (number of immediate successors) on that level. Let $L_n$ be the 
$n$th level of $T$.

First let $T_0 = T$. Now suppose that $T_n$ has been constructed. Let $X_n$ be the set of
members $x$ of height $k_n$ in $T_n$ such that $\{y \in T_n: x \le y\}$ is infinite, and subject to 
this have minimum degree (since $T_n$ has $\omega$ levels, at least one node on that level has 
infinitely many points above it). Let $T_{n+1}$ be the subtree obtained by removing all members 
of $T_n$ $\ge$ some member of $L_{k_n} \setminus X_n$. Then $T_{n+1}$ is a subtree of $T_n$ also having
$\omega$ levels. Let $T^* = \bigcap_{n \in \omega} T_n$. It remains to show that $T^*$ is balanced 
and has no leaves.

Consider any $n$. Since each member of $\omega$ is listed infinitely often, $\{m: k_m = n\}$ is
infinite. Let $L_n^m$ be the $n$th level of $T_m$. Then $L_n^0 \supseteq L_n^1 \supseteq L_n^2 \ldots$
is a decreasing sequence of non-empty finite sets, so is eventually constant, and clearly this equals the 
$n$th level of $T^*$. Furthermore every $x$ in the $n$th level of $T^*$ has at least one successor in
every $T_m$, and hence also in $T^*$, so is not a leaf, and for the same reason, its degree is 
constant among nodes on that level. Hence $T^*$ is a balanced subtree of $T$.
  \end{proof} 

Note that it is tempting to try to prove this in a `simpler' way, by pruning the tree successively
in $\omega$ steps, at step 1 pruning branches above nodes on level 1 to ensure that all nodes on
level 1 now have the same degree (the original minimum), then repeating this on level 2 and so on. 
The drawback is that one has to choose which branches to remove at each stage, so AC has been used.
It is easy to show by a Fraenkel--Mostowski model that this method cannot work in general. 

Now let us see how we can refine the method given in the lemma. The aim here is to show that any
finitely branching tree can be written in a canonical way as a union of balanced subtrees. The change in 
the method is that at the induction step, we retain all nodes, but partition them according to degrees 
of nodes on the next level up, and also higher levels.

\begin{theorem} Let $(T, \le)$ be a tree with $\omega$ levels, all finite, with $k$th level $L_k$. 
Then there is a sequence of ordered partitions $((\pi_k, <): k \in \omega)$ such that $\pi_k$ is a partition
of $L_k$, and such that for any $X \in \pi_k$ and $Y \in \pi_{k+1}$ there is $n(X,Y) \in \omega$ such that for
each $x \in X$ the number of successors of $x$ lying in $Y$ is equal to $n(X,Y)$. \label{4.2}  \end{theorem}

\begin{proof} We use the same sequence $(k_n)$ as in the previous proof, in which each member of $\omega$ 
appears infinitely often, and for each $k$ we find a refining sequence of partitions 
$((\pi_k^n, <): n \in \omega)$ of $L_k$, which here is taken to mean that each member of $\pi_k^n$ is a 
union of a convex subset of $\pi_k^{n+1}$. We start with each $\pi_k^0$ 
being the trivial partition of $L_k$ into just one piece.

Now assume that $(\pi_k^n, <)$ has been chosen for each $k$. If $k \neq k_n$, then $\pi_k^{n+1} = \pi_k^n$.
So we just have to define $\pi_k^{n+1}$ (and a suitable ordering of it) in the case that $k = k_n$. For each
$x \in L_k$, and $X \in \pi^n_{k+1}$, consider the sequence of natural numbers $(|X \cap succ(x)|: X \in \pi^n_{k+1})$
of length the cardinality of $\pi^n_{k+1}$ and let $x \sim y$ if the sequences corresponding to $x$ and $y$ are equal.
We then take $\pi^{n+1}_k$ to be the least common refinement of $\pi^n_k$ and the partition determined by $\sim$. We
can (definably) linearly order the finite sequences of natural numbers and hence extend the ordering of $\pi^n_k$ to
$\pi^{n+1}_k \upharpoonright X$ for each $X \in \pi^n_k$ which will thus form a convex subset.

Since $(\pi^n_k: n \in \omega)$ is a refining sequence and $L_k$ is finite, it must be eventually constant, at $\pi_k$ 
say. The relationship between $\pi_k$ and $\pi_{k+1}$ may now be described. Clearly as the sequence has stabilized,
$\sim$ must be trivial on each member of $\pi_k$. So if $X \in \pi_k$, $(|Y \cap succ(x)|: Y \in \pi_{k+1})$ is the same
sequence for each $x \in X$. Since $|succ(x)|$ is the sum of this sequence, this means in particular that each $x \in X$ 
has the same number of successors. And for each $Y \in \pi_{k+1}$, each $x \in X$ has the same number of successors in $Y$,
which we write as $n(X,Y)$ (which is allowed to be 0).     \end{proof} 

From the construction just given, we form what we call a {\em template} $(T^*, \prec)$ from $T$, which is also a finitely 
branching $\omega$-tree, but with extra structure added. The $n$th level of $T^*$ is just $\pi_n$, and if $X \in \pi_n$ and 
$Y \in \pi_{n+1}$, then $X \prec Y$ provided that $n(X,Y) \neq 0$ (and $\prec$ on the whole of $T^*$ is the transitive
closure of these individual relations). The additional structure that $T^*$ has is first of all the linear ordering of 
each level (which was defined in the above construction), and in addition the sequence of positive integers 
$(n(X,Y): Y$ a $\prec$-successor of $X)$ (with ordering derived from that of $\pi_{n+1}$). We note that
$x$ is a leaf of $T$ if and only if $n(X,Y) = 0$ for all $Y$ on the next level, where $x \in X$.

Now since $(T^*, \prec)$ with the ordering of the levels is a finitely branching $\omega$-tree, we can apply K\"onig's
Lemma to obtain an infinite branch. We note that this is legitimate, since we can invoke the ordering of each level to avoid
appealing to AC. Alternatively, we can remark that in fact the set of all nodes of the template can be well-ordered (for
essentially the same reason), so it actually has cardinality $\aleph_0$. Thus, even though the role of the template is to 
help describe potentially non-well-orderable phenomena (the possible $(T,<)$), it can itself be well-ordered.

In fact, if $\sigma = (X_0, X_1, X_2, \ldots)$ is such an infinite branch, meaning that $X_n \prec X_{n+1}$ for each $n$ and
$X_n$ lies in the $n$th level of $T^*$, then the tree induced from $T$ on $\bigcup_{n \in \omega}X_n$ is a balanced 
tree (noting that $X_0$ is trivial since $|L_0| = 1$, so it only admits the trivial partition). The intuition 
is that $T^*$ somehow collects together all the different possibilities for balanced subtrees of $T$.

Abstracting from the the above, a {\em template} is a finitely branching $\omega$-tree in which each level is linearly ordered,
and such that each node is assigned a finite sequence of positive integers whose entries are in bijective correspondence with 
its set of successors. A template {\em encodes} a finitely branching $\omega$-tree $T$ if it arises as $T^*$ from $T$ in the 
above construction. We may spell this out more precisely by saying that $T^*$ encodes $T$ if there are ordered partitions 
$\pi_n$ of the levels of $T$ which correspond to the levels of $T^*$, and for each $x \in X \in \pi_n$ labelled by the 
sequence $(n_i)$, and $Y \in \pi_{n+1}$ such that $Y$ is the $i$th $\prec$-successor of $X$, there are $n_i$ members of $Y$
greater than $x$.

If $(T^*, \prec)$ is a template, and $(X_0, X_1, X_2, \ldots)$ is an $\omega$-branch, then there is a sequence of positive integers
whose $n$th entry is the label at $X_n$ corresponding to $X_{n+1}$. This is called an {\em eventually singleton branch} if 
for all but finitely many $n$, this label is 1.

For trees $T_1$ and $T_2$ having $\omega$ finite levels, we say that $T_1$ is {\em locally embeddable} in $T_2$ if for 
each $n$, the union of the first $n$ levels of $T_1$ is embeddable as a subtree of the first $n$ levels of $T_2$. 
If $T_1$ is locally embeddable in $T_2$, and vice versa, we say that $T_1$ and $T_2$ are {\em locally isomorphic}. 

\begin{lemma} (i) For any template $(T^*, \prec)$ there is a finitely branching $\omega$-tree $(T, <)$ encoded by $(T^*, \prec)$.

(ii) Any two trees encoded by $(T^*, \prec)$ are locally isomorphic.

(iii) If $(T^*, \prec)$ has an eventually singleton $\omega$-branch, then $(T, <)$ has an infinite branch.

(iv) If $(T^*, \prec)$ is a template having no eventually singleton branch, then in a suitable Fraenkel--Mostowski model, 
there is a tree encoded by $(T^*, \prec)$ having no infinite branch. \label{4.3}  \end{lemma}

\begin{proof} (i) Let $(T^*, \prec)$ be a template with $n$th level $P_n$ and positive integers attached as in the
definition. We build $(T,<)$ level by level, using induction. There will be a bijection between $P_n$ and a partition $\pi_n$ of
$L_n$. Choose any root for $T$, and let $\pi_0$ be the trivial singleton partition of $L_0$ which just contains the root, and
$P_0$ corresponds to $\pi_0$.

Assuming that $L_n$ has been chosen, as has a partition $\pi_n$ of $L_n$ which corresponds to $P_n$ under some bijection, for
$X \in P_n$ let $Y_1, \ldots, Y_k$ be the members of $P_{n+1}$ above $X$ under $\prec$, with positive integers $n_1, \ldots, n_k$
attached. Under the bijection given by the induction hypothesis, $X$ may be viewed as a subset of $L_n$. For each $x \in X$, and
$i$ between 1 and $k$, choose $n_i$ new points which together (as $i$ varies) will form $succ(x)$ in $T$. Doing this for all
$x$ determines $L_{n+1}$. The $(n+1)$th level $P_{n+1}$ of $T^*$ is now identified with a partition $\pi_{n+1}$ of $L_{n+1}$ by placing
$y_1$ and $y_2$ in the same piece of $\pi_{n+1}$ provided that for some $X \in \pi_n$, $x_1, x_2 \in X$, and $i$ between 1 and $k$,
$y_1$ and $y_2$ are adjoined corresponding to that value of $i$.

(ii) We observe that we have the freedom to choose the new points adjoined with or without an ordering. Thus if we are (secretly) 
thinking of them as ordered, then the whole of the tree can be well-ordered, and has cardinality $\aleph_0$. If we make no attempt
to order them, then there is no reason why this should be true. So we do not expect the two possibilities to give rise to 
isomorphic trees. It is clear however that they will be locally isomorphic, since their behaviour is identical up to any finite
stage.

(iii) As remarked above, a branch of the template corresponds to a balanced subtree of $T$, and saying that this is eventually singleton
just says that this subtree is eventually just a single branch.

(iv) What we really want to say is that $(T,<)$ has an infinite branch if and only if its template has an eventually singleton branch,
but for the reasons explained this is too strong. Suppose then that $(T^*, \prec)$ is a template and that it encodes $(T, <)$. We form
a Fraenkel--Mostowski model in which $U$ is indexed by the members of $T$, and we let $G$ be the group of automorphisms of $U$ induced by
tree automorphisms of $(T,<)$ which fix each member of $T^*$, and we let $\mathfrak N$ be the FM-model defined using finite supports.
One sees that $G$ acts on each subtree obtained as above from an $\omega$-branch of $T^*$. This is a balanced tree which has
infinitely many non-singleton entries, and so as for the simplest construction given from just one balanced tree, has no infinite
branch in $\mathfrak N$. Since any infinite branch of $T$ would have to arise in this way, there cannot be any at all. \end{proof}

We note that in the absence of AC, two non-isomorphic trees may arise from isomorphic templates (for instance as in
Lemma \ref{4.3}(iv) there can be two trees arising from $T^*$, one of which can be well-ordered, and the other cannot). 
In the next result we give a more precise result explaining what happens. 

\begin{lemma} (i) Assuming the axiom of choice, for any two $\omega$-trees with finite levels, $T_1$ 
is locally embeddable in $T_2$ if and only if it is embeddable in $T_2$.

(ii) Assuming the axiom of choice, $T_1$ and $T_2$ are locally isomorphic if and only if they
are isomorphic.

(iii) $T_1$ and $T_2$ are locally isomorphic if and only if they are elementarily equivalent. \label{4.4}  \end{lemma}

\begin{proof}  
(i) Let $T_1$ be locally embeddable in $T_2$. Let $P_n$ be the family of all embeddings of the union of the
first $n$ levels of $T_1$ into the union of the first $n$ levels of $T_2$. Let $P = \bigcup_{n \in \omega}P_n$
be partially ordered by extension. Then $P$ is a tree with $n$th level $P_n$, each of which is finite, and by
assumption non-empty. By K\"onig's Lemma, $P$ has an infinite branch, and the union of this branch provides an
embedding of $T_1$ into $T_2$.

(ii) This follows by the same proof. 

(iii) Assume first that $T_1$ and $T_2$ are elementarily equivalent. For any $n$ there is a first order sentence
capturing the first $n$ levels of $T_1$ up to isomorphism. More precisely, if $N$ is the number of nodes in these
levels (clearly a finite number), the sentence says that there exist $N$ points, which are ordered in the
correct way, and any other $n$ points which are linearly ordered and form an initial segment are contained in
the $N$ first mentioned. Since $T_1$ and $T_2$ are elementarily equivalent, there must be corresponding $N$ points
of $T_2$, which are ordered in the same way as those in $T_1$, and which form the union of the first $n$ levels.
Hence the unions of the first $n$ levels of $T_1$ and $T_2$ are isomorphic.

Conversely, if $T_1$ and $T_2$ are locally isomorphic, then we can use back-and-forth (an
Ehrenfeucht-Fra\"iss\'e game) to see that they are elementarily equivalent. Suppose that this has $n$
moves. Then player II has a winning strategy, in which he uses given isomorphisms between the unions of
finitely many levels to decide what to do next.          \end{proof}

\vspace{.1in}

Now we shall allow levels to be infinite, and for this we have to modify the definition of `balanced'.
The levels will certainly be Dedekind finite, and usually weakly Dedekind finite, but the point is that 
they will or can carry some structure. The problem here is that in the not AC situation, the isomorphisms
which one might wish to exist between different successor sets may be absent, even though we want to view 
them as the same. For this we use the notion of `equivalence', introduced above.

We now say that a tree with $\omega$ levels and no leaves is {\em balanced} if at each level, the sets of 
successors of the nodes on that level are equivalent. Ideally we would like to show that any tree with 
$\omega$ levels all of which are weakly Dedekind finite has a balanced subtree with the same properties
(analogously to Lemma \ref{4.1}). In the absence of this, we give some examples and constructions which
illustrate what can happen. First of all we analyze the circumstances under which a tree of this kind 
can be shown to be Dedekind finite.

Let $X$ be a countably infinite set, and $G$ a group of permutations of $X$. Recall, as introduced above, 
that for any $A \subseteq X$, its {\em definable closure} $dcl(A)$ is the set of all $x \in X$ fixed by the 
pointwise stabilizer $G_A$ of $A$ in $X$. We say that $dcl$ is {\em locally finite} if $A$ finite implies 
that $dcl(A)$ is finite.
 
\begin{lemma} \label{4.5} Let $G$ be a group of permutations of the set $U$ of atoms in a model $\mathfrak M$
of $FMC$, and let $\mathfrak N$ be Fraenkel--Mostowski model thereby determined using finite supports. Then
$|U| \in \Delta$ in $\mathfrak N$ if and only if $dcl$ for $G$ on $U$ is locally finite.   \end{lemma}

\begin{proof} If $(x_n)$ is a sequence of distinct members of $U$ in $\mathfrak N$ it is supported by a 
finite $A \subseteq U$. Thus any member of $G$ fixing $A$ pointwise also fixes the sequence $(x_n)$,
and hence each $x_n$. Therefore every $x_n$ lies in $dcl(A)$, and so $dcl$ is not locally finite.

Conversely, if $dcl$ is not locally finite, there is finite $A \subseteq U$ for which $dcl(A)$ is
infinite. Let $x_n$ be distinct members of $dcl(A)$. Then $G_A$ fixes each $x_n$ and so the sequence
$(x_n)$ lies in $\mathfrak N$ and therefore $|U| \not \in \Delta$ in $\mathfrak N$. \end{proof}

We now want to apply this result to study models in which there are cardinals in $\Delta \setminus \Delta_5$. 
Let $(X_n: n \in \omega)$ be a sequence of countably infinite sets, and for each $n$ let $G_n$ be a transitive
group of permutations of $X_n$. We let the set of atoms $U$ be (indexed by) the finite sequences of the form 
$(x_i: i < n)$ for $x_i \in X_i$, $n \in \omega$. Then $U$ becomes a tree with $\omega$ levels under the 
relation of extension, and for each $\sigma \in U$ of length $n$, the set $succ(\sigma)$ of successors of 
$\sigma$ is a copy of $X_n$, so we may allow $G_n$ to act on it. In fact, more generally, $G_n$ acts on $U$ 
by fixing all sequences of length $\le n$, and permuting longer sequences just by the $(n+1)$th entry. In the 
Fraenkel--Mostowski model determined from this sequence of sets and group actions, we let $\mathcal G$ be the 
group of tree automorphisms of $U$ generated by the actions of all the $G_n$s as just described. (This is an 
iterated wreath product of the sequence of groups.)

\begin{theorem} \label{4.6} In the Fraenkel--Mostowski model $\mathfrak N$ determined from a sequence of
non-trivial sets $X_n$ and transitive group actions $G_n$, $|U| \in \Delta$ in $\mathfrak N$ if and only if 
$dcl$ is locally finite for each $G_n$ on $X_n$.  \end{theorem}

\begin{proof} If $dcl$ for the action of some $G_n$ is not locally finite on $X_n$, then using the method of 
Lemma \ref{4.5} gives a countably infinite subset of $U$ in the model.

Otherwise we have to show that if $dcl$ is locally finite for each $G_n$ on $X_n$, then so is $dcl$ for 
$\mathcal G$ on $U$ (and then appeal again to Lemma \ref{4.5}). Let $A \subseteq U$ be finite. We can see
that $dcl(A)$ contains the subtree generated by $A$ (i.e. its downward closure), and this is still finite.
So now assuming that $A$ actually is this subtree, we observe that its definable closure is just the union
of the definable closures of its intersections with the successor sets arising. This follows since the
transitivity of each $G_n$ and the non-triviality of $X_n$ guarantee that no members of the definable closure 
of $A$ appear above the levels involving members of $A$. By hypothesis these are all finite, and hence so 
is $dcl(A)$.    \end{proof}

We remark that all the trees arising in such models $\mathfrak N$ are balanced. We now give some examples.

In the first case, each $succ(x)$ is finite. This was fully analyzed earlier in the section.

Usually all the groups $G_n$ will be equal, but this is not necessary for the construction. In 
\cite{Ballesteros}, a construction was given in which each $X_n$ equals the ordered set of rational numbers
under the group of all order-preserving permutations. This was specifically constructed to provide an example of 
a dense rigid chain (meaning it has no non-trivial order-automorphisms) admitting a non-trivial order-preserving 
surjection. In this case, definable closure is trivial (i.e. $dcl(A) = A$).

The `simplest' case is where each $G_n$ is the full symmetric group. All the individual successor sets are 
strictly amorphous (which means that all partitions of the set into infinitely many subsets have just finitely many 
non-singletons). Again definable closure is trivial.

In the cases so far described, the levels are all weakly Dedekind-finite, and one's hope that more complicated objects 
can be somehow described in terms of simpler constituents is realized. We can however give further examples to show 
that this doesn't always happen. Returning to themes given earlier in the paper, each successor set can be the union
of an $\omega$-sequence of pairs. The definable closure of a finite set is just the union of the pairs that it
intersects. If instead of an $\omega$-sequence, we just consider the {\em set} of pairs, which amounts to allowing
the group to permute the pairs (another wreath product) then we instead revert to the amorphous case (this time 
the successor sets are bounded amorphous sets of gauge 2, in the terminology of \cite{Truss3}). In both these
cases, definable closure is non-trivial.

If we take all successor sets to be a weakly 2-transitive tree, then we obtain many more examples in which these sets 
also lie in $\Delta_5 \setminus \Delta_4$. Finally, if we `recycle' one of our trees having $\omega$ levels as itself 
forming the levels of a new such tree, then we obtain a 2-step example, where the levels lie in $\Delta \setminus \Delta_5$.

We conclude this section by posing questions arising. In the first place, can one prove the analogues of 
Lemma \ref{4.1} and Theorem \ref{4.2} for general weakly Dedekind finite level sets? And can one find an example of a set
lying in $\Delta \setminus \Delta_5$ which cannot be written as a tree with $\omega$ levels and `simpler'
successor sets (that is, even applying the final method mentioned iteratively)?

\end{document}